\documentclass{amsart}
\usepackage{graphicx}
\usepackage{amsmath}
\usepackage{amssymb}
\usepackage{oldgerm}
\usepackage[all]{xy}
\usepackage{color}

\newcommand{\Hom}{\operatorname{Hom}\nolimits}

\newcommand{\id}{\operatorname{id}\nolimits}

\newcommand{\Ext}{\operatorname{Ext}\nolimits}

\newcommand{\HH}{\operatorname{HH}\nolimits}

\newcommand{\Z}{\operatorname{Z}\nolimits}

\newcommand{\V}{\operatorname{V}\nolimits}
\newcommand{\s}{\operatorname{\Sigma}\nolimits}

\newcommand{\T}{\operatorname{\mathcal{T}}\nolimits}

\newcommand{\ev}{\operatorname{ev}\nolimits}
\newcommand{\Proj}{\operatorname{Proj}\nolimits}
\newcommand{\supp}{\operatorname{Supp}\nolimits}
\newcommand{\stmod}{\operatorname{\underline{mod}}\nolimits}
\newcommand{\sthom}{\operatorname{\underline{Hom}}\nolimits}

\newtheorem{theorem}{Theorem}[section]

\newtheorem{corollary}[theorem]{Corollary}

\newtheorem{lemma}[theorem]{Lemma}
\newtheorem{proposition}[theorem]{Proposition}

\theoremstyle{definition}
\newtheorem*{definition}{Definition}
\newtheorem{setup}[theorem]{Setup}
\theoremstyle{definition}

\theoremstyle{definition}

\theoremstyle{definition}

\theoremstyle{definition}

\newtheorem*{notation}{Notation}
\theoremstyle{definition}

\newtheorem{remark}[theorem]{Remark}
\theoremstyle{definition}

\theoremstyle{definition}

\DeclareMathOperator{\efg}{Mod_{efg}}
\newcommand{\growth}{\gamma}
\DeclareMathOperator{\CM}{CM}
\DeclareMathOperator{\stCM}{\underline{CM}}

\begin{document}
\title[Cohomological symmetry]{Cohomological symmetry in triangulated categories}
\author[Bergh]{Petter Andreas Bergh}
\author[Oppermann]{Steffen Oppermann}

\address{Petter Andreas Bergh \\ Institutt for matematiske fag \\
  NTNU \\ N-7491 Trondheim \\ Norway}
\email{bergh@math.ntnu.no}

\address{Steffen Oppermann \\ Institutt for matematiske fag \\
  NTNU \\ N-7491 Trondheim \\ Norway}
\email{Steffen.Oppermann@math.ntnu.no}




\begin{abstract}
We give a criterion for cohomological symmetry in a triangulated category. As an application, we show that such cohomological symmetry holds for all pairs of modules over any exterior algebra. 
\end{abstract}

\maketitle

\section{Introduction}

Given a ring $A$ and two finitely generated modules $M$ and $N$, does the equivalence
$$\Ext_A^n(M,N)=0 \; \forall n \gg 0 \quad \Longleftrightarrow \quad \Ext_A^n(N,M)=0 \; \forall n \gg 0$$
hold? Obviously, for most rings, this kind of symmetry cannot hold: it would for example imply that $\Ext_A^n(M,A)=0$ for all $n \gg 0$ and every finitely generated module $M$. Nevertheless, there are large classes of rings having such symmetry. The first nontrivial such class of rings was given by Avramov and Buchweitz in \cite{AvramovBuchweitz}, where it was shown that symmetry holds over all commutative local complete intersection rings. In \cite{HunekeJorgensen}, Huneke and Jorgensen extended this class to include the so-called AB-rings, which are certain commutative local Gorenstein rings. For noncommutative rings, the most prominent class known to satisfy symmetry are the group algebras of finite groups, as shown by Mori in \cite{Mori}.

In this paper, we show that symmetry holds over exterior algebras. Partial results already exist in the literature. By \cite[Corollary~4.9 and Theorem~6.6]{Mori}, symmetry holds for exterior algebras on odd dimensional vector spaces, and for all graded modules over any such algebra. We prove that it holds for all modules over all exterior algebras.

Our strategy towards this goal is as follows: in Section~\ref{sect.triang} we investigate cohomological symmetry for triangulated categories with a central ring action. We show that, under certain conditions, cohomological symmetry is a local property (see Theorem~\ref{GlobalSymmetry}). In Section~\ref{sect.applic} we apply this result, first to Gorenstein algebras in general (Theorem~\ref{theorem.gorenstein}), and then to exterior algebras in particular (see Theorem~\ref{SymmetryExterior} and Corollary~\ref{ExtSymmetryExterior}).

\section{Symmetry in triangulated categories} \label{sect.triang}

Let $k$ be a field. Throughout this section, we fix a triangulated $k$-category $\T$ with suspension functor $\s$. Given objects $X$ and $Y$, we define
$$\Hom^*_{\T}(X,Y) \stackrel{\text{def}}{=} \bigoplus_{n \in \mathbb{Z}} \Hom_{\T}(X, \s^nY).$$
We start by recalling the notions of graded center, central ring actions and support of objects. For more details, see \cite{BIKO}.

The \emph{graded center} of $\T$, denoted $\Z^*( \T )$, is defined by $\Z^n( \T )$ being the collection of all natural transformations $1_{\T} \xrightarrow{f} \s^n$ satisfying $f_{\s X} = (-1)^n \s f_X$ for every object $X$. (We may think of $\Z^*( \T )$ as being a graded algebra, however note that in general there is no reason for the collections involved to be sets.) Straightforward induction then gives $f_{\s^t X} = (-1)^{nt} \s^t f_X$ for every $t \in \mathbb{Z}$. For objects $X$ and $Y$, a homogeneous central element $f \in \Z^{|f|}( \T )$ acts both from the right and from the left on the graded abelian group $\Hom^*_{\T}(X,Y)$ in a natural way. Explicitly, if $g \in \Hom_{\T}(X, \s^{|g|}Y)$ is a homogeneous element, then the scalar multiplications with $f$ are given by
\begin{eqnarray*}
fg & \stackrel{\text{def}}{=} & ( \s^{|g|} f_Y ) \circ g \\
gf & \stackrel{\text{def}}{=} & ( \s^{|f|} g ) \circ f_X. 
\end{eqnarray*} 
Both these morphisms are elements in $\Hom_{\T}(X, \s^{|f|+|g|}Y)$, and they are actually equal up to a sign. Namely, since $f$ is a natural transformation from the identity functor on $\T$ to the functor $\s^{|f|}$, the diagram
$$\xymatrix{
X \ar[r]^{g} \ar[d]^{f_X} & \s^{|g|}Y \ar[d]^{f_{\s^{|g|}Y}} \\
\s^{|f|} X \ar[r]^{\s^{|f|}g} & \s^{|f]+|g|}Y }$$
commutes. Then since $f_{\s^{|g|}Y} = (-1)^{|f||g|} \s^{|g|} f_Y$, we see that
$$fg = (-1)^{|f||g|} gf.$$
This shows two things. First, by taking $g$ to be the coordinate map $X \xrightarrow{g_X} \s^{|g|}X$ at $X$ of another homogeneous central element of $\Z^*( \T )$, we see that the graded center is graded-commutative. Second, for all objects $X$ and $Y$ the graded abelian group $\Hom^*_{\T}(X,Y)$ is a graded $\Z^*( \T )$-module, on which the left and right scalar actions coincide up to sign. Note also that the map
\begin{eqnarray*}
\Z^*( \T ) & \to & \Hom^*_{\T}(X,X) \\
f & \mapsto & f_X
\end{eqnarray*}  
is a homomorphism of graded algebras.

Let $H = \oplus_{n=0}^{\infty} H_n$ be a graded-commutative $k$-algebra (positively graded). Then $H$ \emph{acts centrally} on $\T$ if there exists a homomorphism $H \to \Z^*( \T )$ of graded algebras. Typical examples are the cohomology ring of a finite group acting on the stable module category, and the Hochschild cohomology ring of a finite dimensional algebra acting on the derived category of modules. Combining with the coordinate maps, we see that the central ring action gives, for each object $X$, a homomorphism
$$H \xrightarrow{\varphi_X} \Hom^*_{\T}(X,X)$$
of graded algebras. Thus, for every pair of objects $X$ and $Y$, the graded group $\Hom^*_{\T}(X,Y)$ becomes a graded left/right $H$ module through $\varphi_Y$ and $\varphi_X$, and these scalar actions coincide up to sign.

For a graded-commutative algebra $H$, the \emph{even part} $H^{\ev} = \oplus_{n=0}^{\infty} H_{2n}$ is commutative. Therefore any central ring action gives rise to an action of a commutative algebra concentrated in even degrees.

\begin{setup} \label{setup.triang_supp}
Throughout the rest of this section, let $R$ be a connected commutative graded $k$-algebra of finite type concentrated in even degrees acting centrally on $\T$. (Here ``connected'' means $R_0 = k$, and ``of finite type'' means that $R$ is finitely generated as an algebra over $k$.) We set $\mathbb{X} = \Proj R$, the set of homogeneous prime ideals of $R$ not containing $R_+ = \oplus_{n > 0} R_n$, with its natural scheme structure.
\end{setup}

\begin{definition}
A graded $R$-module $M$ is \emph{eventually finitely generated} if there is a number $n_0$ such that the submodule $M_{\geq n_0} = \oplus_{n \geq n_0} M_n$ is finitely generated. We denote by $\efg R$ the category of eventually finitely generated $R$-modules.
\end{definition}

\begin{notation}
Let $M$ be a graded $R$-module.
\begin{enumerate}
\item The \emph{support} of $M$ is
\[ \supp_{\mathbb{X}} M = \{ \mathfrak{p} \in \mathbb{X} \mid M_{\mathfrak{p}} \neq 0 \}, \]
where $M_{\mathfrak{p}}$ denotes the localization of $M$ with respect to the prime ideal $\mathfrak{p}$.
\item The \emph{rate of growth} of $M$ is
\[ \growth M = \inf \{ \alpha \in \mathbb{N} \cup \{ 0 \} \mid \exists C \colon \dim_k M_n \leq Cn^{\alpha-1}  \text{ for } n \gg 0 \}. \]
\end{enumerate}
\end{notation}

\begin{remark} \label{rem.cut_doesnt_matter} 
(1) Let $M$ be a graded $R$-module, $n_0 \in \mathbb{Z}$, and $\mathfrak{p} \in \mathbb{X}$. Then the short exact sequence
\[ 0 \to M_{\geq n_0} \to M \to M/M_{\geq n_0} \to 0 \]
gives rise to a short exact sequence
\[ 0 \to  (M_{\geq n_0})_{\mathfrak{p}} \to M_{\mathfrak{p}} \to \underbrace{(M/M_{\geq n_0})_{\mathfrak{p}}}_{= 0} \to 0, \]
in which the rightmost term vanishes since $\mathfrak{p}$ does not contain all elements of $R$ of positive degree. In particular, $\supp_{\mathbb{X}} M = \supp_{\mathbb{X}} M_{\geq n_0}$.

(2) If $M$ is a finitely generated $R$-module, then $\supp_{\mathbb{X}} M$ is a closed subset of $\mathbb{X}$. By (1), the same claim holds if we just require $M \in \efg R$.

(3) \label{rem.cx_dim} If $M \in \efg R$, then 
$$\growth M = \left\{ \begin{array}{ll} 0 & \text{ if } \supp_{\mathbb{X}} M = \emptyset \\ 1 + \dim \supp_{\mathbb{X}} M & \text{ if } \supp_{\mathbb{X}} M \neq \emptyset. \end{array} \right. $$
\end{remark}

We now apply the general notion of support to the setup we are interested in.

\begin{notation}
By abuse of notation, for $X, Y \in \T$ we write
\[ \supp_{\mathbb{X}} (X, Y) = \supp_{\mathbb{X}} \Hom_{\T}^*(X, Y), \]
and
\[ \supp_{\mathbb{X}} (X) = \supp_{\mathbb{X}} \Hom_{\T}^*(X, X). \]
\end{notation}

In the definition of support of a pair of objects, the order is essential: there are examples where $\supp_{\mathbb{X}} (X,Y)$ does not equal $\supp_{\mathbb{X}}(Y,X)$, even when $\Hom_{\T}^*(X, Y) \oplus \Hom_{\T}^*(Y, X) \in \efg R$ (see, for example, \cite[Example, Section 4]{Bergh3}). The main concern in this paper is whether the statement
\[ \Hom_{\T}(X, \s^nY)=0 \; \forall n \gg 0 \quad \Longleftrightarrow \quad \Hom_{\T}(Y, \s^nX)=0 \; \forall n \gg 0 \]
holds: under the finiteness condition, by Remark~\ref{rem.cut_doesnt_matter}(3), this is equivalent to asking whether
$$\supp_{\mathbb{X}}(X,Y) = \emptyset \quad \Longleftrightarrow \quad \supp_{\mathbb{X}}(Y,X) = \emptyset$$
holds.

The following lemma records some of the elementary properties of triangulated support which we will make use of in the further discussion.

\begin{lemma} \label{ElementaryProperties}
In the situation of Setup~\ref{setup.triang_supp}, let $X$ and $Y$ be two objects of $\T$. Then the following hold.
\begin{enumerate}
\item $\supp_{\mathbb{X}} (X,Y) \subseteq \supp_{\mathbb{X}} (X) \cap \supp_{\mathbb{X}} (Y)$.
\item For a homogeneous element $r \in R$, complete the map $\varphi_X(r)$ to a distinguished triangle
$$X \xrightarrow{\varphi_X(r)} \s^{|r|}X \to X /\!\!/r \to \s X$$
in $\T$.
\begin{itemize}
\item[(a)] If $\Hom^*_{\T}(X,Y) \in \efg R$, then also $\Hom^*_{\T}(X /\!\!/r, Y) \in \efg R$, and
\[ \supp_{\mathbb{X}}  (X /\!\!/r,Y) = \V_{\mathbb{X}}(r) \cap \supp_{\mathbb{X}} (X,Y). \]
\item[(b)] If $\Hom_{\T}^*(Y, X) \in \efg R$, then $\Hom^*_{\T}(Y, X /\!\!/r) \in \efg R$, and
\[ \supp_{\mathbb{X}} (Y,X /\!\!/r) = \V_{\mathbb{X}}(r) \cap \supp_{\mathbb{X}} (Y,X). \]
\end{itemize}
\end{enumerate}
\end{lemma}

\begin{proof}
Property (1) follows from the fact that $R$ acts on $\Hom^*_{\T}(X,Y)$ through both the ring homomorphisms $R \xrightarrow{\varphi_X} \Hom^*_{\T}(X,X)$ and $R \xrightarrow{\varphi_Y} \Hom^*_{\T}(Y,Y)$. The second property follows from the proof of \cite[Theorem 3.3]{BIKO}, see also \cite[Proposition 3.10]{AvramovIyengar}.
\end{proof}

Next, we define support symmetry at a given (closed) point $\mathfrak{p} \in \mathbb{X}$.

\begin{definition} \label{def.local_symm}
In the situation of Setup~\ref{setup.triang_supp}, let $\mathfrak{p} \in \mathbb{X}$. We say that $\T$ satisfies \emph{symmetry at $\mathfrak{p}$} if for all objects $X$ and $Y$ with $\supp_{\mathbb{X}}(X) \cup \supp_{\mathbb{X}}(Y) \subseteq \{ \mathfrak{p} \}$, the equality 
$$\supp_{\mathbb{X}} (X,Y) = \supp_{\mathbb{X}} (Y,X)$$
holds.
\end{definition}

\begin{remark}
In the setup of the previous definition, the inclusions $\supp_{\mathbb{X}} (X, Y) \subseteq \{ \mathfrak{p} \}$ and $\supp_{\mathbb{X}} (Y, X) \subseteq \{ \mathfrak{p} \}$ hold, by Lemma~\ref{ElementaryProperties}(1). Therefore, the assertion is equivalent to
\[ \supp_{\mathbb{X}} (X, Y) = \emptyset \quad \Longleftrightarrow \quad \supp_{\mathbb{X}} (Y, X) = \emptyset. \]
If $\Hom_{\T}^*(X, Y) \oplus \Hom_{\T}^*(Y, X) \in \efg R$, then this in turn is equivalent to
\[ \Hom(X, \s^n Y) = 0 \; \forall n \gg 0 \quad \Longleftrightarrow \quad \Hom(Y, \s^n X) = 0 \; \forall n \gg 0, \]
by Remark~\ref{rem.cut_doesnt_matter}(3). In other words, symmetry at a point means cohomological symmetry for objects whose support is this point.
\end{remark}

We now show that if $\T$ satisfies symmetry at a closed point, then this point ``occurs symmetrically" in the supports of objects.

\begin{proposition} \label{prop.OccursSymmetrically}
In the situation of Setup~\ref{setup.triang_supp}, suppose $\T$ satisfies symmetry at a closed point $\mathfrak{p} \in \mathbb{X}$, and let $X$ and $Y$ be objects in $\T$ satisfying $\Hom^*_{\T}(X,Y) \oplus \Hom^*_{\T}(Y,X) \in \efg R$.
Then 
\[ \mathfrak{p} \in \supp_{\mathbb{X}}(X,Y) \quad \Longleftrightarrow \quad \mathfrak{p} \in \supp_{\mathbb{X}}(Y,X). \]
\end{proposition}

\begin{proof}
Since $\mathfrak{p}$ is a homogeneous ideal, there are homogeneous elements $r_1, \dots, r_c$ in $R$ such that $\mathfrak{p} = (r_1, \ldots, r_c)$. Thus $\{ \mathfrak{p} \} = V_{\mathbb{X}}(r_1) \cap \cdots \cap V_{\mathbb{X}}(r_c)$. By Lemma~\ref{ElementaryProperties}(1), for each $i$
\[ \supp_{\mathbb{X}} (X  /\!\!/ r_i, Y) = \V_{\mathbb{X}}(r_i) \cap \supp_{\mathbb{X}}(X, Y), \]
and hence
\[ \mathfrak{p} \in \supp_{\mathbb{X}} (X  /\!\!/ r_i, Y) \quad \Longleftrightarrow \quad \mathfrak{p} \in \supp_{\mathbb{X}} (X, Y). \]
Similarly
\[ \mathfrak{p} \in \supp_{\mathbb{X}} (Y, X  /\!\!/ r_i) \quad \Longleftrightarrow \quad \mathfrak{p} \in \supp_{\mathbb{X}} (Y, X). \]
Iterating this process as long as necessary, we see that we can replace $X$ by an object whose support is contained in $\{ \mathfrak{p} \}$. Similarly, we can replace $Y$ by an object whose support is contained in $\{ \mathfrak{p} \}$. The result then follows from the definition of symmetry at $\mathfrak{p}$.
\end{proof}

Next, we prove that a local-global principle holds for support symmetry.

\begin{theorem}\label{GlobalSymmetry}
Assume in the situation of Setup~\ref{setup.triang_supp} that for any $X, Y \in \T$, the graded module $\Hom^*_{\T}(X,Y)$ lies in $\efg R$. Then the following are equivalent.
\begin{enumerate}
\item For any closed point $\mathfrak{p} \in \mathbb{X}$, the category $\T$ satisfies symmetry at $\mathfrak{p}$.
\item For any $X, Y \in \T$, the equality
\[ \supp_{\mathbb{X}} (X, Y) = \supp_{\mathbb{X}} (Y, X) \]
holds.
\end{enumerate}
\end{theorem}

\begin{proof}
The implication (2) $\Rightarrow$ (1) is trivial, since (1) is just a special case of (2). Assume therefore that (1) holds. It follows from Proposition~\ref{prop.OccursSymmetrically} that the closed points of $\supp_{\mathbb{X}} (X, Y)$ and of $\supp_{\mathbb{X}} (Y, X)$ coincide. Since $\mathbb{X}$ is of finite type over $k$, the closed points determine every closed subset.
\end{proof}

As an immediate corollary, local support symmetry implies global vanishing symmetry.

\begin{corollary} \label{GlobalVanishingSymmetry}
Assume in the situation of Setup~\ref{setup.triang_supp} that for any $X, Y \in \T$, the graded module $\Hom^*_{\T}(X,Y)$ lies in $\efg R$. Moreover, suppose that for any closed point $\mathfrak{p} \in \mathbb{X}$, the category $\T$ satisfies symmetry at $\mathfrak{p}$. Then
\[ \Hom_{\T}(X, \s^nY)=0 \; \forall n \gg 0 \quad \Longleftrightarrow \quad \Hom_{\T}(Y, \s^nX)=0 \; \forall n \gg 0 \]
for all objects $X$ and $Y$.
\end{corollary}

\section{Applications} \label{sect.applic}

In this final section, we turn to modules over algebras. Throughout, all modules considered are finitely generated left modules. 

Let $k$ be a field and $c \ge 1$ an integer. The \emph{exterior algebra} on $c$ generators is the algebra
$$\Lambda = k \langle X_1, \dots, X_c \rangle / (X_i^2, \{ X_iX_j+X_jX_i \}_{i \neq j}).$$
It is local, selfinjective, and of dimension $2^c$. Our aim is to prove that 
\[ \Ext_{\Lambda}^n(M,N)=0 \; \forall n \gg 0 \quad \Longleftrightarrow \quad \Ext_{\Lambda}^n(N,M)=0 \; \forall n \gg 0 \]
holds for all $\Lambda$-modules $M$ and $N$, by applying the results from the previous section.

\subsection{Gorenstein algebras}

Let us first more generally consider a finite dimensional Gorenstein algebra $\Lambda$ (that is $\id \Lambda_{\Lambda}$ and $\id \,_{\Lambda}\Lambda$ are both finite). We work in the stable category of maximal Cohen-Macaulay modules, denoted $\stCM \Lambda$.  Its objects are the maximal Cohen-Macaulay $\Lambda$-modules, that is, the modules satisfying $\Ext_{\Lambda}^i(M, \Lambda) = 0$ for $i > 0$. The morphism spaces are of the form
$$\sthom_{\Lambda} (M,N) = \Hom_{\Lambda}(M,N)/P_{\Lambda}(M,N),$$
where $P_{\Lambda}(M,N)$ is the subspace of $\Hom_{\Lambda}(M,N)$ consisting of the homomorphisms that factor through projective modules. The stable category of maximal Cohen-Macaulay modules is triangulated, with suspension $\s$ given by $\s^{-1} = \Omega_{\Lambda}$, the syzygy functor. Ordinary cohomology can be interpreted as cohomology in $\stCM \Lambda$: for maximal Cohen-Macaulay modules $M$ and $N$, the vector spaces $\Ext_{\Lambda}^n(M,N)$ and $\sthom_{\Lambda}(M, \Sigma^n N)$ are isomorphic for all $n \ge 1$.

By a classical result of Gerstenhaber (cf.\ \cite{Gerstenhaber}), the Hochschild cohomology ring $\HH^*( \Lambda )$ of any finite dimensional algebra $\Lambda$ is graded-commutative. Moreover, by \cite[Section 10]{Solberg} and \cite[Section 3]{BuchweitzFlenner}, it acts centrally on the bounded derived category of $\Lambda$-modules, via tensor products. In the case when $\Lambda$ is Gorenstein, this central ring action induces a central ring action on $\stCM \Lambda$ (cf.\ \cite[Remark 5.1]{BIKO}), since $\stCM \Lambda \cong {\rm D^b}(\Lambda) / {\rm per}\, \Lambda$.

Our aim in this subsection is to to apply Theorem~\ref{GlobalSymmetry} to Gorenstein algebras. We therefore assume the following finiteness condition.

\begin{setup} \label{setup.fg_mod}
\sloppy Let $\Lambda$ be a finite dimensional Gorenstein algebra. Denote by $\HH^{\ev}(\Lambda)$ the even part of the Hochschild cohomology ring $\HH^*( \Lambda )$, and write $\mathbb{X} = \Proj \HH^{\ev}(\Lambda)$. Assume $\HH^{\ev}(\Lambda)$ is a finitely generated $k$ algebra, and for all maximal Cohen-Macaulay $\Lambda$-modules $M$ and $N$, the $\HH^{\ev}(\Lambda)$-module $\Ext_{\Lambda}^*(M, N) = \oplus_{n \geq 1} \Ext_{\Lambda}^n(M, N)$ is finitely generated.
\end{setup}

We shall see in the next subsection that these assumptions are satisfied when $\Lambda$ is an exterior algebras. The following result then reduces the question of cohomological symmetry to the periodic modules, i.e.\ the modules
$M$ for which there is an integer $n \in \mathbb{N}$ such that $\Omega^n_{\Lambda} M \cong M$.

\begin{theorem} \label{theorem.gorenstein}
In the setup above, the following are equivalent.
\begin{enumerate}
\item For all periodic $\Lambda$-modules $M$ and $N$ we have
\[ \Ext_{\Lambda}^n(M, N) = 0 \; \forall n \gg 0 \quad \Longleftrightarrow \quad \Ext_{\Lambda}^n(N, M) = 0 \; \forall n \gg 0. \]
\item For all $\Lambda$-modules $M$ and $N$ we have
\[ \Ext_{\Lambda}^n(M, N) = 0 \; \forall n \gg 0 \quad \Longleftrightarrow \quad \Ext_{\Lambda}^n(N, M) = 0 \; \forall n \gg 0. \]
\item For all $\Lambda$-modules $M$ and $N$ we have
\[ \growth(\Ext_{\Lambda}^*(M, N)) = \growth(\Ext_{\Lambda}^*(N, M)). \]
\item For all $\Lambda$-modules $M$ and $N$ we have
\[ \supp_{\mathbb{X}} (M, N) = \supp_{\mathbb{X}} (N, M). \]
\end{enumerate}
\end{theorem}

\begin{proof}
The implications (3) $\Rightarrow$ (2) $\Rightarrow$ (1) are clear, whereas the implication (4) $\Rightarrow$ (3) follows from Remark~\ref{rem.cut_doesnt_matter}(3). It remains therefore to show (1) $\Rightarrow$ (4).

Let $g = \id \,_{\Lambda}\Lambda$. Then for any module $M$, the $g$-th syzygy $\Omega^g_{\Lambda} M$ is a maximal Cohen-Macaulay module. For $n > g$, there are isomorphisms 
\[ \Ext_{\Lambda}^n(M, N) \cong \Ext_{\Lambda}^{n-g}(\Omega^g_{\Lambda} M, N) \cong \Ext_{\Lambda}^n(\Omega^g_{\Lambda} M, \Omega^g_{\Lambda} N). \]
Since the support only depends on the behavior of $\Ext_{\Lambda}^n(-,-)$ for large $n$, we may assume that $M$ and $N$ are maximal Cohen-Macaulay modules. Since 
$$\Hom_{\stCM \Lambda}(M, \s^n N) \cong \Ext_{\Lambda}^n(M, N)$$
for $n \ge 1$, the assumptions of Theorem~\ref{GlobalSymmetry} are satisfied (with $\T = \stCM \Lambda$).

By \cite[Theorem~2.3]{Bergh1}, any module $M$ with $\growth( \Ext_{\Lambda}^*(M,M)) \leq 1$ is eventually periodic, and it follows that any such module which in addition is maximal Cohen-Macaulay is periodic. Now we can complete the proof as follows, where we use the notation of Theorem~\ref{GlobalSymmetry} for $\T = \stCM \Lambda$:
\begin{align*}
\text{(1)} & \Longrightarrow \forall M, N \in \CM \Lambda \text{ with } \growth( \Ext_{\Lambda}^*(M,M) \oplus \Ext_{\Lambda}^*(N,N)) \leq 1 \colon \\
& \qquad \qquad \Ext_{\Lambda}^n(M, N) = 0 \; \forall n \gg 0 \; \Leftrightarrow \; \Ext_{\Lambda}^n(N, M) = 0 \; \forall n \gg 0 \\
& \overset{\text{\ref{rem.cut_doesnt_matter}(3)}}{\Longrightarrow} \forall M, N \in \CM \Lambda \text{ with } \dim (\supp_{\mathbb{X}} (M,M) \cup \supp_{\mathbb{X}} (N,N)) = 0 \colon \\
& \qquad \qquad \supp_{\mathbb{X}} (M, N) = \emptyset \; \Leftrightarrow \; \supp_{\mathbb{X}} (N, M) = \emptyset \\
& \overset{\text{\ref{GlobalSymmetry}}}{\Longrightarrow} \forall M, N \in \CM \Lambda \colon \supp_{\mathbb{X}} (M, N) = \supp_{\mathbb{X}} (N, M). \qedhere
\end{align*}
\end{proof}

\subsection{Exterior algebras}

In this final subsection we specialize to the case when $\Lambda$ is an exterior algebra. In this case, it was proved in \cite{BerghOppermann} that the Hochschild cohomology ring $\HH^*(\Lambda)$ is of finite type, and that $\Ext_{\Lambda}^*(M,N)$ is a finitely generated $\HH^*(\Lambda)$-module for all $\Lambda$-modules $M$ and $N$. The even part $\HH^*(\Lambda)^{\ev}$ of $\HH^*(\Lambda)$ is therefore a commutative algebra of finite type, and the finiteness condition of Setup~\ref{setup.fg_mod} holds: for all $M$ and $N$, the $\HH^*(\Lambda)^{\ev}$-module $\Ext^*_{\Lambda}(M,N)$ is finitely generated over $\HH^*(\Lambda)^{\ev}$.

Let $M$ and $N$ be two $\Lambda$-modules, and $\mathbb{X} = \Proj \HH^{\ev}(\Lambda)$.
By definition, the support $\supp_{\mathbb{X}}(M,N)$ is the Zariski closed subset
\begin{align*}
 \supp_{\mathbb{X}}(M,N) & = \{ \mathfrak{p} \in \mathbb{X} \mid \Hom^*_{\stmod \Lambda}(M,N)_{\mathfrak{p}} \neq 0 \} \\
& = \{ \mathfrak{p} \in \mathbb{X} \mid \Ext^*_{\Lambda}(M,N)_{\mathfrak{p}} \neq 0 \}
\end{align*}
of $\mathbb{X}$ (note that the two sets on the right coincide by Remark~\ref{rem.cut_doesnt_matter}(1)). Consequently, our support defined via triangulated categories is in this setup just the projective version of the classical support varieties defined by Snashall and Solberg in \cite{SnashallSolberg}.

Our main result for exterior algebras, which is just an application of Theorem~\ref{theorem.gorenstein}, shows that the order of the modules is irrelevant.

\begin{theorem} \label{SymmetryExterior}
If $\Lambda$ is an exterior algebra, then
\[ \supp_{\mathbb{X}}(M,N) = \supp_{\mathbb{X}}(N,M) \]
for all $\Lambda$-modules $M$ and $N$.
\end{theorem}

Before we give the proof, let us record some immediate corollaries.

\begin{corollary} \label{cor.same_growth}
If $\Lambda$ is an exterior algebra, then for all $\Lambda$-modules $M$ and $N$
\[ \growth( \Ext_{\Lambda}^*(M, N)) = \growth( \Ext_{\Lambda}^*(N, M)). \]
\end{corollary}

\begin{proof}
Follows from Theorem~\ref{theorem.gorenstein} and Theorem~\ref{SymmetryExterior}.
\end{proof}

\begin{corollary}\label{ExtSymmetryExterior}
If $\Lambda$ is an exterior algebra, then the following are equivalent for all $\Lambda$-modules $M$ and $N$:
\begin{enumerate}
\item $\Ext_\Lambda^n(M,N) = 0$ for $n \gg 0$,
\item $\Ext_\Lambda^n(M,N) = 0$ for $n \ge 1$,
\item $\Ext_\Lambda^n(N,M) = 0$ for $n \gg 0$,
\item $\Ext_\Lambda^n(N,M) = 0$ for $n \ge 1$.
\end{enumerate}
\end{corollary}

\begin{proof}
The equivalences $(1) \Leftrightarrow (2)$ and $(3) \Leftrightarrow (4)$ follow from \cite[Theorem 3.6]{Bergh3}. The equivalence $(1) \Leftrightarrow (3)$ is a special case of Corollary~\ref{cor.same_growth}.
\end{proof}

We finish this paper by giving the proof of Theorem~\ref{SymmetryExterior}.

\begin{proof}[Proof of Theorem~\ref{SymmetryExterior}]
Suppose first that both $M$ and $N$ are periodic, and that $\Ext_{\Lambda}^n(M, N) = 0$ for $n \gg 0$. It follows from the periodicity that $\sthom_{\Lambda}(M, \Omega_{\Lambda}^n N) = 0$ for all $n \in \mathbb{Z}$. The Auslander-Reiten formula (cf.\ \cite{AuslanderReiten}) then implies that
\begin{eqnarray*}
0 & = & D \sthom_{\Lambda}( M, \Omega_{\Lambda}^n(N)) \\
& \simeq & \sthom_{\Lambda}( \Omega_{\Lambda}^{1+n}(N), \Omega_{\Lambda}^2( {_{\nu}M})) \\
& \simeq & \sthom_{\Lambda}( \Omega_{\Lambda}^{n-1}(N), {_{\nu}M})
\end{eqnarray*}
for all $n \in \mathbb{Z}$, where ${_{\nu}M}$ is $M$ twisted by the Nakayama automorphism $\nu$ of $\Lambda$. Let $c$ be the number of generators of our exterior algebra $\Lambda$, i.e.\ $\Lambda$ is the exterior algebra on the $c$ generators $X_1, \dots, X_c$. By \cite[Lemma 3.1]{Bergh2}, if $c$ is odd then $\nu$ is the identity, whereas if $c$ is even then $\nu$ maps each $X_i$ to $-X_i$. In the odd case, the twisted module ${_{\nu}M}$ therefore coincides with the twisted module $\tau M$ defined in \cite[Section 2]{Eisenbud}, and then, since $M$ is periodic, there is an isomorphism ${_{\nu}M} \simeq \Omega_{\Lambda}(M)$ by \cite[Theorem 2.2]{Eisenbud}. Consequently, we see that
$$0 = \sthom_A( \Omega_A^n(N),M) \quad \forall n \in \mathbb{Z},$$
regardless of the parity of $c$.

Summing up, we have shown that for periodic $M$ and $N$, the implication
\[ \Ext_{\Lambda}^n(M, N) = 0 \; \forall n \gg 0 \quad \Longrightarrow \quad \Ext_{\Lambda}^n(N, M) = 0 \; \forall n \geq 1 \]
holds. The claim of the theorem now follows by the implication (1) $\Rightarrow$ (4) of Theorem~\ref{theorem.gorenstein}.
\end{proof}

\end{document}